\documentclass{amsart}
\usepackage{amsmath}
\usepackage{mathtools}
\usepackage{amsthm}
\usepackage[all,dvips]{xy}
\usepackage[]{pstricks,pst-node,pst-plot}
\usepackage{graphicx}

\def\PP{{\mathbb P}}

\def\C{{\mathcal C}}
\def\D{{\mathcal D}}

\def\H{{\mathcal H}}
\def\H{{\mathcal H}}

\def\M{{\mathcal M}}

\def\tH{\widetilde{\mathcal H}}
\def\tC{\widetilde{\mathcal C}}
\def\tP{\widetilde{\PP}}
\def\tY{\widetilde{Y}}
\def\bM{\overline{\mathcal M}}
\def\bM{\overline{\M}}

\newtheorem{theorem}{Theorem}[section]
\newtheorem{lemma}[theorem]{Lemma}
\newtheorem{proposition}[theorem]{Proposition}

\newtheorem{definition-lemma}[theorem]{Definition-Lemma}

\newtheorem{conclusion}[theorem]{Conclusion}
\theoremstyle{definition}
\newtheorem{definition}[theorem]{\bf Definition}

\theoremstyle{remark}

\newtheorem{diagram}[theorem]{\bf Diagram}

\def\vandaag{\number\day\space\ifcase\month\or
 januari\or februari\or  maart\or  april\or mei\or juni\or  juli\or
 augustus\or  september\or  oktober\or november\or  december\or\fi,
\number\year}
\def\today{\ifcase\month\or
 Jan\or Febr\or  Mar\or  Apr\or May\or Jun\or  Jul\or
 Aug\or  Sep\or  Oct\or Nov\or  Dec\or\fi
 \space\number\day, \number\year}

\begin{document}

\title[Divisors on Hurwitz Spaces]{Divisors on Hurwitz Spaces: 
An Appendix to `The Cycle Classes of Divisorial Maroni Loci'}
\author{Gerard van der Geer $\,$}
\address{Korteweg-de Vries Instituut, Universiteit van
Amsterdam, Postbus 94248, 1090 GE Amsterdam,  The Netherlands}
\email{G.B.M.vanderGeer@uva.nl}
\author{$\,$ Alexis Kouvidakis}
\address{Department of Mathematics and Applied Mathematics, 
University of Crete,
GR-70013 Heraklion, Greece}
\email{kouvid@uoc.gr}
\begin{abstract}
The Maroni stratification on the Hurwitz space of degree $d$ covers of genus $g$
has a stratum that is a divisor only if $d-1$ divides $g$. Here we
construct a stratification on the Hurwitz space that is analogous to the
Maroni stratification, but has a divisor
for all pairs $(d,g)$ with $d \leq g$ with a few exceptions 
and we calculate the divisor class of an extension
of these divisors to the compactified Hurwitz space.
\end{abstract}
\maketitle
\begin{section}{Introduction}\label{sec:intro}
The Hurwitz space $\H_{d,g}$ of simply-branched covers 
of genus $g$ and degree $d$ carries a stratification 
named after Maroni (\cite{Maroni}) that is defined as follows. 
If $\gamma: C \to {\PP}^1$ is a simply-branched cover one takes
the dual of the cokernel of the natural map
$$
{\mathcal O}_{\PP^1} \to \gamma_{*}{\mathcal O}_{C}
$$
which is a vector bundle of rank $d-1$ on the projective line, hence is
isomorphic to ${\mathcal O}(a_1)\oplus \cdots \oplus {\mathcal O}(a_{d-1})$ for some
$(d-1)$-tuple $\alpha=(a_1,\ldots,a_{d-1})$, where we assume that the $a_i$
are non-decreasing. 
The loci of covers $\gamma: C \to {\PP}^1$ with fixed $\alpha$ are
the strata. 

It is known (see \cite{Ballico}) that for general
$\gamma: C \to {\PP}^1$ of genus
$$
g= k(d-1)+s \qquad \hbox{\rm with $0\leq s \leq d-2$}
$$
the tuple $\alpha$ takes the form $(k+1,\ldots,k+1,k+2,\ldots,k+2)$ 
with $s$ enties equal to $k+2$. 
Only the case with $s=0$ yields a Maroni stratum that is a divisor
(see \cite{C-M}  and \cite[Thm.\ 1.15]{Patel}).

In this paper we show how to define for the case that the genus
$g$ is not divisible by $d-1$ a stratification that has a stratum
that is a divisor for $g\geq d$ under exclusion of a few cases. 
If $d-1$ divides $g$ then this reduces to the stratification of 
Maroni loci. It
uses instead of the cokernel of 
${\mathcal O}_{\PP^1} \to \gamma_{*}{\mathcal O}_{C}$ 
the cokernel of a natural map
$$
{\mathcal O}_{\PP^1} \to \gamma_*{\mathcal O}_C(D)\, ,
$$
where $D$ is an appropriately chosen divisor of degree $s$
with support in the ramification locus of $\gamma$. 
The cycle classes of an extension of these divisors to the
compactified Hurwitz space $\overline{\H}_{d,g}$ 
can be calculated by using a global-to-local evaluation map 
$p^*p_* V \to V$
of a vector bundle $V$ on an extension of the ${\PP}^1$-fibration
$p: {\PP} \to \H_{d,g}$
to the compactified Hurwitz space that is trivial on the generic fibre
of $p$. The calculation and the answer are completely analogous to case
of the cycle classes of the Maroni divisors calculated in \cite{vdG-K}.
The cycle classes are given in terms of an explicit sum of boundary divisors. 
\bigskip

\noindent
{\sl Acknowledgement.} The first author thanks YMSC of Tsinghua University
for the hospitality enjoyed there.
\end{section}
\begin{section}{The Setting}\label{setting}
We recall the setting from \cite{vdG-K}.
We denote by $\overline{\H}_{d,g}$ the compactified Hurwitz space 
of admissible covers of degree $d$ and genus $g$. We have
$$
\overline{\H}_{d,g}-{\H}_{d,g}=\cup_{j,\mu} S_{j,\mu},
$$
where the $S_{j,\mu}=S_{b-j,\mu}$ are divisors indexed by $2\leq j \leq b-2$
and a partition $\mu=(m_1,\ldots,m_n)$ of $d$. These divisors can be reducible,
but a generic point corresponds to an admissible cover $\gamma: C \to P$
where $P$ is a genus $0$ curve consisting of two components $P_1, P_2$ 
of genus $0$ intersecting in one point $Q$ with $j_1=j$ or $b-j$ 
branch points on $P_1$ 
(resp.\ $j_2=b-j$ or $j$ branch points on $P_2$) 
and the inverse image $\gamma^{-1}(Q)$
consists of $n$ points $Q_1,\ldots,Q_n$ on $C$ with ramification indices 
$m_1,\ldots,m_n$. Since $\overline{\H}_{d,g}$ is not normal 
we normalize it and this results in a smooth stack $\widetilde{\H}_{d,g}$. 
We then have a diagram
$$
\begin{xy}
\xymatrix{
{\tC} \ar[r]^c \ar[d]_{\varpi}& {\bM}_{0,b+1}\ar[d]^{\pi_{b+1}} \\
{\tH}_{d,g} \ar[r]^{h} & {\bM}_{0,b} \\
}
\end{xy}
$$
where $\widetilde{\C}$ is the universal curve and
${\bM}_{0,b}$ is the moduli space of stable $b$-pointed curves of genus $0$
and $\pi_{b+1}$ is the map that forgets the $(b+1)$st point.
With $\PP$ the fibre product of ${\bM}_{0,b+1}$ and ${\tH}_{d,g}$ over
${\bM}_{0,b}$ we have a diagram
$$
\begin{xy}
\xymatrix{
{\tC} \ar[r]^{\alpha} \ar[rd]_{\varpi}& {\PP}\ar[d]^{\varpi'} 
\ar[r]^{c'} &{\bM}_{0,b+1}\ar[d]^{\pi_{b+1}} \\
& {\tH}_{d,g} \ar[r]^{h} & {\bM}_{0,b} \\
}
\end{xy}
$$
We now work over a base $B$ (it can be $\H_{d,g}$, ${\tH}_{d,g}$ or 
often a $1$-dimensional base). 
But we shall suppress the index $B$.
Note that normalization commutes with base change.
In \cite[Lemma 4.1]{vdG-K} we showed that ${\tC}$ and ${\PP}$ 
have only singularities of type $A_k$. We resolve the singularities 
of ${\PP}$ obtaining a space $\widetilde{\PP}$ and then let
$Y$ be the normalization of ${\tC}\times_{\PP}{\tP}$ and let
$\widetilde{Y}$ be the resolution of singularities of $Y$. 
We then find the basic  diagram as in \cite{vdG-K}:

\begin{diagram}\label{basicdiagram}
$$
\begin{xy}
\xymatrix{
\widetilde{Y} \ar@/^/[drr]^{\rho}  \ar[rd]^{\nu} \ar[rdd]_{\tilde{\pi}} \\
& Y \ar[r] \ar[d]^{\pi} & {\tC} \ar[d] \ar[rdd]^{\varpi} \\
& {\tP} \ar[r] \ar[rrd]_p & {\PP} \ar[rd] \\
&&& B
}
\end{xy}
$$
\end{diagram}

We observe that the finite map $\pi: Y \to \tP$ is flat 
as  $Y$ is Cohen-Macaulay and $\tP$ is smooth. Actually, $Y$
has rational singularities only.

\end{section}
\begin{section}{Constructing Divisors}
Let $\D$ be an effective divisor on $Y$ of relative degree $s$ over $B$, 
supported on the sections. This is a Cartier divisor since the sections 
do not intersect the singular locus and so ${\mathcal O}(\D)$ is a line bundle on $Y$. 
Therefore it follows 
that ${\pi}_*{\mathcal O}(\D)$ is a locally free sheaf on ${\tP}$. 
We denote by $\tilde{\D}$ the proper transform of $\D$ under the resolution
map $\nu$.
Since $\nu_*{\mathcal O}(\tilde{{\D}})= \nu_*{\mathcal O}_{\tilde{Y}}
\otimes {\mathcal O}({\mathcal D})$, we conclude by
 \cite[Lemma 4.4]{vdG-K} and the fact that $R^j\nu_*{\mathcal O}_{\tY}=0$
for $j\geq 1$, 
that $\tilde{\pi}_*{\mathcal O}(\tilde{\D})= \pi_*{\mathcal O}({\D})$. We can use the 
restriction of $\pi_*(\mathcal O(\D))$ to the open part over $\H_{d,g}$
to define a stratification by type of the bundle on $\PP^1$ just
as for the Maroni stratification. We are interested in the case we get
a divisor.

For a divisor $\D$ we have an inclusion
$$
\iota_{\D} : {\mathcal O}_{\tilde{\PP}} \to \tilde{\pi}_*{\mathcal O}(\tilde{\D})\, .
$$
Note that the image $\iota_{\D}(1)$ of the section $1$ 
is a nowhere vanishing section of
$\tilde{\pi}_*{\mathcal O}(\tilde{\D})$.

We now introduce the vector bundle that we use
to define a stratification.
\begin{definition} 
We let ${\mathcal K}_{\D}$ be the cokernel of $\iota_{\D}$. We define
$V_{\mathcal D} := {\mathcal K}_{\mathcal D}^{\vee}$ as the dual
${\mathcal O}_{\tilde{\PP}}$-module. 
Since $\iota_{\D}(1)$ is a nowhere vanishing section of
$\tilde{\pi}_* {\mathcal O}(\tilde{\D})$ the sheaf 
${\mathcal K}_{\mathcal D}$ is locally free of rank $d-1$ 
on $\tilde{\PP}$ and therefore
so is $V_{\mathcal D}$.
\end{definition}

We start with a lemma which  follows immediately from
the Riemann-Roch theorem.

\begin{lemma}
\label{lemma_minimal}
Let $U = \oplus_{i=1}^r {\mathcal O}(a_i)$ be a vector bundle of rank $r$ and degree $n$ on $\PP^1$.
Suppose that $-1 \leq a_1 \leq \cdots \leq a_r$. Then $h^0(V ) = r + n$. Moreover, this is
the minimum dimension for the space of sections of a vector bundle of rank $r$ and degree $n$ on ${\mathbb P}^1$.
\end{lemma}
Given a cover $\gamma: C \to {\PP}^1$ and an effective divisor $D$ 
of degree $s$ supported on the  
ramification divisor of $\gamma$, we write
$$
\gamma_*{\mathcal O}_C (D) \cong {\mathcal O}_{{\mathbb P}^1}(-a_0) \oplus {\mathcal O}_{{\mathbb P}^1}(-a_1) \oplus \cdots
\oplus {\mathcal O}_{{\mathbb P}^1}(-a_{d-1}) ,
$$
with $a_0\leq a_1\leq \cdots \leq a_{d-1}$.  
Note that $\sum_{i=1}^{d-1}a_i=(k+1)(d-1)$.

We let $\gamma : C \to \PP^1$ represent a general point of 
${\mathcal H}_{d,g}$. We want to determine the numbers $a_i$.
We know by results of Ballico \cite{Ballico} that
$$
\gamma_*{\mathcal O}_C \cong {\mathcal O}_{\PP^1} \oplus {\mathcal O}_{\PP^1}(-(k + 1))^{\oplus d-1-s}
\oplus  {\mathcal O}_{\PP^1}(-(k + 2))^{\oplus s} .
$$
Thus we get
\[
\gamma_*{\mathcal O} \otimes {\mathcal O}(k) \cong {\mathcal O}_{\PP^1}(k) \oplus 
{\mathcal O}_{\PP^1}(-1)^{\oplus d-1-s} \oplus  {\mathcal O}_{\PP^1}(-2)^{\oplus s} .
\]
so that $h^0(k \, g^1_d) = k + 1$ and $h^1(k \, g^1_d)=s$.

\begin{proposition}\label{propositionkg1d+D}
For $\gamma : C \to \PP^1$ a general point of ${\mathcal H}_{d,g}$  
we can choose an effective divisor $D$ of degree $s$ 
supported on the ramification locus $R$ of $\gamma$ satisfying
\[
h^0(D) = 1 \quad \hbox{and}\quad  h^0(k \, g^1_d+D) = k + 1 .
\]
\end{proposition}
\begin{proof}
Note that since for a general  $\gamma $ we have $h^0(k\, g^1_d)=k+1$, 
the first condition is implied by the second.
This is because for any effective divisors $D_1$, $D_2$ on $C$ we have  
that   $h^0(D_1+D_2)\geq h^0(D_1)+h^0(D_2)-1$.
Observe that $b-(2g - 2) = 2\, d$. 
We consider the linear system $|K_C - k\, g^1_d|$.
For reasons of degree we can choose a ramification point 
$p_1$ which is not a base point of $|K_C - k\, g^1_d|$ 
and this gives $h^0(K_C - k\, g^1_d-p_1)=s-1$. 
Then by the same degree argument we can find a ramification point 
$p_2$ such that it is not a base point
of $|K_C - k \, g^1_d-p_1|$ such that $h^0(K_C - k\, g^1_d-p_1-p_2)=s-2$. 
Repeating the argument we arrive at a divisor of degree $s$ 
supported on the ramification locus
such that $h^0(K_C - k\, g^1_d -D)=0$, 
hence by duality $h^1(k \, g^1_d+D)=0$. By Riemann-Roch we have
$h^0(k \, g^1_d + D) = k + 1$.
\end{proof}
Now if we choose $D$ as in 
Proposition \ref{propositionkg1d+D} 
we have $h^0(D)=1$ and therefore $a_0=0$ and
$a_i>0$ for $i\geq 1$. Moreover $h^0(k \, g^1_d+D)=k+1$ and  since
\[
\gamma_*{\mathcal O}_C (D) \otimes {\mathcal O}(k) \cong {\mathcal O}_{{\PP}^1}(k) \oplus 
{\mathcal O}_{{\PP}^1}(k-a_1) \oplus \cdots
\oplus {\mathcal O}_{{\PP}^1}(k-a_{d-1}) 
\]
this implies that $a_i \geq k+1$ for all 
$i = 1, \cdots , d-1$. Since the $a_i$ add up to $(k+1)(d-1)$
we conclude that all $a_i = k + 1$.
\begin{conclusion}
\label{conclusion_1}
If we choose $\gamma$  and $D$ as in Proposition \ref{propositionkg1d+D} 
then the dual of the cokernel of
${\mathcal O}_{{\PP}^1} \to \gamma_*{\mathcal O}_C(D)$ has type
${\mathcal O}(k +1)^{\oplus d-1}$.
\end{conclusion}

We now want to see that our degeneracy locus is non-empty in 
the open Hurwitz space $\H_{d,g}$. 
According to Ohbuchi \cite{Ohbuchi} only so-called acceptable $(d-1)$-tuples
$(a_1, \cdots , a_{d-1})$ can occur as the indices of the 
dual of the cokernel of ${\mathcal O}_{{\PP}^1} \to \gamma_*{\mathcal O}_C $; 
here acceptable is defined as follows, see \cite{Patel}.
\begin{definition}
\label{conditions}
A non-decreasing $(d-1)$-tuple of natural numbers 
$(a_1,\cdots, a_{d-1})$ with $\sum_{i=1}^{d-1}a_i= b/2$
is said to be acceptable for $(d, g)$ if the $a_i$ satisfy
\begin{enumerate}
\item  $a_1 \geq b/d(d - 1)$;
\item $a_{d-1}\leq b/d$;
\item $a_{i+1}-a_1\leq a_1$.
\end{enumerate}
\end{definition}

We now consider the unique acceptable $(d - 1)$-tuple $\alpha$ 
for which with $a_1=k$
and for which the sum
\[
\sum_{i=1}^{d-1}(d-i)a_i
\]
is maximal. This means that this $\alpha$ 
is the most balanced $(d-1)$-tuple with $a_1 = k$.
In \cite{Coppens} Coppens proved the following existence result, 
see also \cite{Patel}.

\begin{theorem}\label{Coppens}
Let $a_1$ be an integer satisfying (2) of Definition \ref{conditions}. 
If $\alpha $ is the unique acceptable $(d - 1)$-tuple 
$(a_1, \cdots, a_{d-1})$ for which the sum
$\sum_{i=1}^{d-1}(d-i)a_i$ is maximal, 
then the locus in ${\H}_{d,g}$ of covers
$\gamma : C \to  {\PP}^1$ with invariant $\alpha$ is not empty.
\end{theorem}
We apply this in the case $a_1 = k$ and 
$\sum_{i=1}^{d-1}(d-i)a_i=(k + 1)(d - 1) + s$. We find the
following maximizing $(d - 1)$-tuples for the $a_i$ of 
$V = \oplus _{i=1}^{d-1}{\mathcal O}(a_i)$
 with $a_1 = k$.

\begin{lemma}\label{lemma_3_cases}
If $a_1 = k$ we have the following maximizing acceptable $(d-1)$-tuples:
\begin{enumerate}
\item if $s\leq d - 4$ the maximizing sequence is 
$(k, (k + 1)^{d-s-3}, (k + 2)^{s+1})$;
\item if $s = d -3$ the maximizing sequence is 
$(k, (k + 2)^{d-2})$, with $g\neq 2(d-2)$;
\item if $s = d - 2$ the maximizing sequence is 
$(k, (k + 2)^{d-3}, k + 3)$, with $g\neq 2d-3$  and $(d,g)\neq (3,5)$.
\end{enumerate}
 \end{lemma}

We consider the three cases of Lemma \ref{lemma_3_cases}.
\begin{lemma}
 \label{lemma_3_cases_2}
Let $\gamma : C \to \PP^1$ be a general cover of type (1), (2) or (3) of 
Lemma \ref{lemma_3_cases}. Then there exists a divisor of degree $s$ 
with support in the ramification locus of $\gamma$ 
such that $h^0(k\, g^1_d) = h^0(k \, g^1_d + D)$ and $h^0(D) = 1$.
\end{lemma}
\begin{proof}
In the first case we consider a general cover 
$\gamma : C \to \PP^1$ of this type. Then
we have by \cite{Ballico}
\[
\gamma_*{\mathcal O}_C = {\mathcal O} \oplus {\mathcal O}(-k) \oplus {\mathcal O}(-(k + 1))^{\oplus d-s-3}
\oplus  {\mathcal O}(-(k + 2))^{\oplus s+1}
\]
so that $h^0(k\, g^1_d) = k + 2$ and accordingly $h^1(k \, g^1_d) = s + 1$. 
We  follow the argument of Proposition \ref{propositionkg1d+D}. 
The argument for the cases (2) and (3) is similar.
\end{proof}
For this pair $(\gamma,D)$ the conditions $h^0(D)=1$, $h^0(k g^1_d + D) = k + 2$ imply that for
\[
 \gamma_*{\mathcal O}(D) = \oplus_{i=0}^{d-1} {\mathcal O}(-a_i)
\]
we must have $a_0=0$, $a_1 = k$ and $a_i \geq k + 1$ for $i = 2, \cdots , d -1$. Then the twisted (with
${\mathcal O}(-k - 1)$) dual of the cokernel of ${\mathcal O}_{\PP ^1} \to  \gamma_*{\mathcal O}_C(D)$
 has the form
\[ \oplus_{i=1}^{d-1}{\mathcal O}(b_i)\;\; \mbox{with}\;\, b_1 = -1 \;\; \mbox{and} \;\, 
b_i \geq 0\;\; \mbox{for}\;\, i = 2, \cdots, d- 1
\]
and thus is not balanced with a space of sections of minimum dimension
(see Lemma \ref{lemma_minimal}). 
The other two cases are dealt with similarly.

\begin{conclusion}\label{conclusion_2}
Let $\gamma : C \to \PP^1$ be a general cover such that the dual 
of the cokernel of
${\mathcal O}_{\PP^1} \to  \gamma_*{\mathcal O}_C$ has type as in 
Lemma \ref{lemma_3_cases}. Then there exists
a divisor $D$ of degree $s$ with support in the 
ramification divisor of $\gamma $ such that the dual of the cokernel of
${\mathcal O}_{\PP^1} \to  \gamma_*{\mathcal O}_C(D)$ has unbalanced type  as above.
\end{conclusion}

We define now our divisors that are analogous to the Maroni divisors. 
We consider the pullback of the diagram \ref{basicdiagram} to the open base
$B = {\mathcal H}_{d,g}$ and choose on $Y_{B}$ a reduced divisor ${\mathcal D}$ consisting of any $s$
sections of $\tilde{\pi}|_{B}$ and we let $V_{\D}$ 
be the dual of the cokernel of the natural map
$\iota_{\D}|_{B}: {\mathcal O}_{{\tilde Y}_{B}}\to 
\tilde{\pi}_*{\mathcal O}({\D}) $ and tensor
$V_{\mathcal D}$ with a line bundle $M$ on $\tilde{\PP}_{B}$ such that 
$V^{\prime}_{\D} = V_{\D} \otimes M$ 
has zero degree on the generic fibre of $\pi$.
The sheaf $p_*V^{\prime}_{\mathcal D}$ is reflexive 
(see \cite{Hartshorne}) and therefore it is a vector bundle on an open $U$
of ${\mathcal H}_{d,g}$  with complement of  codimension $\geq 3$.  

Consider an open subset $U'$ of $U$ such that $p_*V_{\D}^{\prime}$
is a trivial of rank $r$, that is,  isomorphic to ${\mathcal O}_{U'}^r$. 
Choose $r$ generating sections $s_1,\ldots,s_r$ of $p_*V_{\D}^{\prime}$ 
on $U'$. 
If we consider their pull backs under $p$ and 
restrict these for $x \in U'$ to $H^0(p^{-1}(x),V_{\D}|p^{-1}(x))$
then they generate the stalk of $V_{\D}^{\prime}$ at a point of 
$p^{-1}(x)$ if and only if $V_{\D}^{\prime}|p^{-1}(x)$ 
is the trivial bundle of rank $r$
on $p^{-1}(x)={\PP}^1$;  indeed, 
 $V_{\D}^{\prime}|p^{-1}(x))=\oplus_i {\mathcal O}(a_i)$ with $\sum_i a_i=0$
and if some $a_i$ are negative, then these sections cannot generate it since
these do not see the ${\mathcal O}(a_i)$ with $a_i$ negative; if 
$V_{\D}^{\prime}|p^{-1}(x)={\mathcal O}_{\PP^1}^r$ by Grauert's theorem they
will generate it (see e.g. \cite[III,Cor.\ 12.9]{HartshorneAG}).
Thus we arrive at the following result.
\begin{theorem}
\label{theorem_generalized_divisor}
Suppose that $g = k(d-1) + s$ with $0\leq s \leq d - 2$ and 
$3\leq d \leq g$, satisfying the conditions that
$g\neq 2d-3$, $g\neq  2d-4$ and $(d,g)\neq (3,5)$. 
Then the vanishing locus of the determinant
of the evaluation map 
${\rm ev} : p^*p_*V^{\prime}_{\D} \to V^{\prime}_{\D} $ defines a
non-empty divisor $\mathfrak{d}_{\D}$ on ${\mathcal H}_{d,g}$.
\end{theorem}
\begin{proof} 
We take a cover as in Conclusion \ref{conclusion_1}. 
Then there is a point in ${\H}_{d,g}$
representing this cover and such that the divisor $D$ of 
Conclusion \ref{conclusion_1} is given by the restriction
of the above ${\D}$ to the corresponding fiber. 
This is because the points of ${\mathcal H}_{d,g}$ parametrize
simply branched coverings with ordered branch points, and hence the 
ramification points are ordered too.  
We  conclude  that $V^{\prime}_{\mathcal D}$ is trivial 
on the generic fibre. 
But by Conclusion \ref{conclusion_2} above 
we see,
for similar reasons,  that the locus in ${\H}_{d,g}$,  
where $V^{\prime}_{\mathcal D}$ restricted to a
fibre is non-trivial is not empty. 
By Lemma \ref{lemma_minimal} and Grauert's theorem
both the above loci belong to $U$.  
Thus the evaluation map on $p^{-1}U$ is a map of vector bundles 
of the same rank with degeneracy locus which is not the whole space
nor empty. 
Therefore it defines a divisor on $p^{-1}U$ which extends uniquely 
to a divisor on ${\H}_{d,g}$.
\end{proof}

\end{section}
\begin{section}{Cycle Classes}
In this section we indicate how to calculate the classes of 
closed divisors in $\overline{\H}_{d,g}$ that extend the above
defined  divisors $\mathfrak{d}_{\D}$. 
We proceed as in \cite{vdG-K} for the Maroni case.

We take the standard line bundle $M$ constructed in 
\cite[Definition 6.2]{vdG-K} 
and we let $V^{\prime}_{\D} =V_{\D} \otimes M$.
We then want to find a divisor $A_{\mathcal D}$ 
supported on the singular fibres of $p: \tilde{Y} \to B$ such that
$c_1(V^{\prime}_{\mathcal D}) - A_{\D}$ is a pullback under 
$p$ from the base~$B$. In \cite[Section 6]{vdG-K} we constructed
such a divisor $A_{\mathcal D}$ in the case where the divisor 
${\D}$ is zero. This is done locally on the base $B$. We
restrict to $1$-dimensional bases $B$ and consider the fibre of
$p$ over a points $s$ where $B$ intersects the boundary.
The fibre there consists of a chain of rational curves.
In the present case the
degrees of $V^{\prime}$, the bundle used in \cite{vdG-K} 
and corresponding to $\D=0$,
 and $V^{\prime}_{\D}$  
differ at a chain $R_0=P_1,R_1, \ldots, R_m=P_2$ only
at $R_0$ and $R_m$.  Adapting the result gives the following analogue
of \cite[Conclusion 6.5]{vdG-K}.
\begin{proposition}
For each irreducible component $\Sigma$ of 
$S_{j,\mu}$ we define
\[ c_{\Sigma,{\mathcal D}}= d-n-2(r-{\mathcal D} \cdot P_2).
\]
If we let
\[
A^{\Sigma}_{\mathcal D}=-{1\over 2}
\sum_{i=0}^{m-1}((m-i)\, c_{\Sigma,{\mathcal D}}-\delta_i)R_i^{\Sigma}
\]
then the degree of $A^{\Sigma}_{\mathcal D}$ on 
$R_i^{\Sigma}$ equals the degree of $V^{\prime}_{\mathcal D}$
on $R_i^{\Sigma}$.
\end{proposition}

The formalism described in \cite{vdG-K} shows that if $Q$ denotes
the degeneracy locus of the evaluation map ${\rm ev}: p^*p_*V^{\prime}_{\D}
\to V_{\D}^{\prime}$ we get that 
$$
c_1(Q)+p^*R^1p_*V_{\D}^{\prime} +(-A_{\D})_{\rm sh},
$$
where the index ${\rm sh}$ 
denotes the shift as in \cite[Definition 3.7]{vdG-K},
is an effective class on $\tP$ that is a pull back under $p$, and
pulling it back under a section of $p$ gives us an effective class
$\overline{\mathfrak{d}}_{\D}$
which extends the divisor $\mathfrak{d}_{\D}$.

To calculate the class $\overline{\mathfrak{d}}_{\D}$
of the degeneracy locus we wish to use 
the formula of \cite[Theorem 3.10]{vdG-K}. For this we need
first a the following lemma (see 
\cite[Proposition 10.1]{vdG-K}).

\begin{lemma}
\label{lemma_chern_classes}
Let ${\mathcal L}$ be a line bundle on $\tilde{Y}$ with first Chern class 
$\ell$. If $U$ denotes the ramification divisor of $\tilde{\pi}$ we have
$$
\begin{aligned}
&c_1(\tilde{\pi}_*{\mathcal L})=\tilde{\pi}_*\ell-\frac{1}{2}\tilde{\pi}_*(U), 
\\
&c_2(\tilde{\pi}_*{\mathcal L})-c_2(\pi_* {\mathcal O}_{\tilde{Y}})=
\frac{1}{2} ((\tilde{\pi}_*\ell)^2-\tilde{\pi}_* (\ell^2)) 
-\frac{1}{2}(\tilde{\pi}_*U\cdot \tilde{\pi}_*\ell-\tilde{\pi}_*(U\cdot \ell)). 
\end{aligned}
$$
\end{lemma}
We denote the sections of $p$ by $\sigma_i$ ($i=1,\ldots,b$)
and their images by $\Xi_i$ ($i = 1, \ldots, b$). These
are divisors on $\tilde{\PP}$. 
Over the section $\Xi_i$ we have a ramification divisor $U_i$. 
If ${\D}_1$ is a reduced divisor with support on the sections 
$\cup_i\Xi_i$ we let ${\D}$ be the corresponding divisor
with support on the ramification divisor of $\tilde{\pi}$. 
Then we have
\[
\tilde{\pi}_*{\D} = {\D}_1, \;\;\;\tilde{\pi}^*{\D}_1 = 
2\, {\D} + \Gamma_{\D} ,
\]
with $\Gamma_{\D}$ disjoint from ${\D}$. 
Another relation that we will use with $W = \tilde{\pi}_*(U)$ is
$$
U \cdot {\D} = \frac{1}{2} W \cdot {\D}_1=\frac{1}{2}{\D}_1^2 .
$$

Using Lemma \ref{lemma_chern_classes} we thus find
$$
c_2(\tilde{\pi}_*{\mathcal O}_{\tilde{Y}}({\D})) = 
c_2(\tilde{\pi}_*{\mathcal O}_{\tilde{Y}}),\quad
c^2_1(\tilde{\pi}_*{\mathcal O}_{\tilde{Y}}({\D}))= 
c^2_1(\tilde{\pi}_*{\mathcal O}_{\tilde{Y}}) .
$$
Therefore, when we apply \cite[Theorem 3.10]{vdG-K}  
in order to calculate the class $\overline{\mathfrak{d}}_{\D}$ 
the only thing that differs from the case treated there
is the choice of the  line bundle $A_{\D}$. This affects only the 
definition of $c_{\Sigma,{\D}}$
and therefore the class  is calculated by the same formula given in 
\cite[Theorem 8.3]{vdG-K} 
with $c_{j,\mu}$ determined by $c_{\Sigma,{\D}}$. 

We can twist the bundle $V^{\prime}_{\D}$ by a line bundle $N$
corresponding to a divisor on $\tP$ with support on the boundary of $\tP$.
This results in a divisor class $\overline{\mathfrak{d}}_{\D,N}$ 
on $\overline{\H}_{d,g}$
extending the divisor $\mathfrak{d}_{\D}$ on $\H_{d,g}$.
We thus arrive as in \cite[Section 9]{vdG-K}
at the following theorem.

\begin{theorem}\label{correction1}
Suppose that $3 \leq d \leq g$ and 
$g\neq 2d-3$, $g\neq  2d-4$ and ($d,g)\neq (3, 5)$. 
Let $\Sigma$ be an irreducible component of the boundary $S_{j,\mu}$ of
$\overline{\mathcal{H}}_{d,g}$. Then
the coefficient $\sigma_{j,\mu}$ of $\Sigma$
of the locus $\cap_N \overline{\mathfrak{d}}_{\D,N}$ is  equal to
$$
\begin{aligned}
 m(\mu)\left(\frac{1}{12}(d-\sum_{\nu=1}^{n(\mu)} \frac{1}{m_{\nu}})+
\frac{j(b-j)(d-2)}{8(b-1)(d-1)}\right)
-\frac{1}{8(d-1)}\sum_{i=1}^{m(\mu)} (\delta_{i-1}-\delta_i)^2 &\\
-\frac{d-1}{2} \left( \frac{m(\mu)}{4} - \sum_{i=1}^{m(\mu)} (e_{i-1}-e_i)^2 
\right) \,& . \\
\end{aligned}
$$
\end{theorem}
For the numbers $m(\mu)$, $\delta_i$ and $e_i$ we refer to \cite{vdG-K}.
Note that the rational numbers $e_i$ that come from rounding off a rational
solution to an integral one 
depend on the constants $c_{j,\mu}$ that occur above.

In a similar way one can obtain analogues of the theorems of Sections 10 and
11 of \cite{vdG-K} by adding to $\D$ an effective divisor $Z$ supported on the
boundary of $\tY$. One can then define an effective class
$\overline{\mathfrak{d}}_{\D,N,Z}$ that extends the class of 
$\mathfrak{d}_{\D}$. One may ask whether
$$
\cap_{N,Z} \overline{\mathfrak{d}}_{\D,N,Z}
$$
is the class of the Zariski closure of $\mathfrak{d}_{\D}$ on $\overline{\H}_{d,g}$.

The classes $\mathfrak{d}_{\D}$ depend on $\D$, but using the monodromy one sees
that their images on the Hurwitz space with unordered branch points do not.
\end{section}


 \end{document}